\documentclass[12pt,a4paper]{amsart}
\usepackage{mathrsfs}
\usepackage{color}
\newtheorem{theo+}              {Theorem}           [section]
\newtheorem{prop+}  [theo+]     {Proposition}
\newtheorem{coro+}  [theo+]     {Corollary}
\newtheorem{lemm+}  [theo+]     {Lemma}
\newtheorem{exam+}  [theo+]     {Example}
\newtheorem{rema+}  [theo+]     {Remark}
\newtheorem{defi+}  [theo+]     {Definition}

\newenvironment{theorem}{\begin{theo+}}{\end{theo+}}
\newenvironment{proposition}{\begin{prop+}}{\end{prop+}}
\newenvironment{corollary}{\begin{coro+}}{\end{coro+}}
\newenvironment{lemma}{\begin{lemm+}}{\end{lemm+}}

\usepackage{amsthm}
\theoremstyle{plain} \theoremstyle{remark}
\newtheorem{remark}{Remark}

\def \r{\mbox{${\mathbb R}$}}
\def\E{/\kern-1.0em \equiv }

\title{}
\author{Ye-Lin Ou$^{*}$ }

\address{Department of Mathematics,\newline\indent
Texas A $\&$ M University-Commerce,\newline\indent Commerce, TX
75429, U S A.\newline\indent E-mail:yelin.ou@tamuc.edu}

\thanks{$^{*}$ This work was supported by a grant from the Simons Foundation ($\#427231$, Ye-Lin Ou).}
\begin{document}

\title[Conformal biharmonic and $k$-polyharmonic maps ]{Some classifications of  conformal biharmonic and  $k$-polyharmonic maps}

\thanks{}
\date{07/21/21}

\maketitle
\section*{Abstract}
\begin{quote}  
We give a complete classification of  local and global conformal biharmonic maps between any two space forms by proving that a  conformal map  between two space forms is proper biharmonic  if and only if the dimension is $4$, the domain is flat, and  it is a restriction of a M\"obius transformation.
 We also show that proper $k$-polyharmonic conformal maps between Euclidean spaces exist if and only if the dimension is $2k$ and they are precisely the restrictions of  M\"obius transformations. This provides infinitely many simple examples of proper $k$-polyharmonic maps with nice geometric structure.
{\footnotesize  }
\end{quote}
\section{Introduction}

Biharmonic maps are a generalization of harmonic maps which include harmonic functions, geodesics, minimal isometric immersions (i.e., minimal submanifols), and Riemannian submersions with minimal fibers as special cases. Biharmonic map equation is a system of 4-th order nonlinear partial differential equations. Although no general theory of the existence of solutions to such a difficult system is available, there are many interesting results on the study of biharmonic maps with additional geometric constraints. For a more detailed background on biharmonic maps,  including basic examples and properties of biharmonic maps, some progress on biharmonic submanifolds (i.e., biharmonic isometric immersions), biharmonic conformal maps, biharmonic maps into spheres, biharmonic maps with symmetry, Liouville type and unique continuation theorems for biharmonic maps, we refer the reader to the recent book \cite{OC}.

The following are some known facts about conformal biharmonic maps between space forms of the same dimension.

\begin{itemize}
\item[(A)] Baird-Kamissoko \cite{BK} : The inversion in sphere, $\phi:\r^{m}\to \r^{m}, \phi(x)=\frac{x}{|x|^2}$, is biharmonic if and only if $m=4$.
\item[(B)] Loubeau-Ou \cite{LO}: (i) The conformal map given by the identity map $1: (B^m, \delta_{ij})\to (B^m, \frac{4\delta_{\alpha\beta}}{(1-|x|^2)^2})\equiv H^m$ is a proper biharmonic map if and only if $m=4$. (ii) The conformal map given by the identity map $1: (\r^m, \delta_{ij})\to (\r^m, \frac{4\delta_{\alpha\beta}}{(1+|x|^2)^2})\equiv S^m\setminus\{N\}$ is a proper biharmonic map if and only if $m=4$.  
\item[(C)] Montaldo-Oniciuc-Ratto \cite{MOR} Gives a classification of rotationally symmetric biharmonic conformal  maps between space forms viewed as warped product model spaces.
\item[(D)] Baird-Ou \cite{BO}: (i) Every conformal map $\phi: \r^4\to \r^4$ with  $\phi(x)= b+\frac{kA(x-a)}{|x-a|^{\epsilon}},\; a, b\in \r^m,\; A\in O(m), \epsilon\in \{0, 2\}$, is biharmonic. For the case $\epsilon=0$, it is also harmonic whilst for the case $\epsilon=2$, every map is proper biharmonic.
(ii) Every conformal map $\phi: \r^4\to S^4\setminus\{N\}\equiv (\r^4, \frac{4\delta_{\alpha\beta}}{(1+|x|^2)^2})$ with  $\phi(x)= b+\frac{kA(x-a)}{|x-a|^{\epsilon}},\; a, b\in \r^m,\; A\in O(m), \epsilon\in \{0, 2\}$, is a proper biharmonic map. (iii) There is no proper biharmonic map among the M\"obius transformations $S^4\to S^4$.
\item[(E)] Ou-Chen \cite{OC} (Corollary 11.12): Proper biharmonic conformal maps between domains of $4$-dimensional space forms exist only in the cases: $\r^4\supset U\to M^4(c)$.
\end{itemize}

In this note, we give a complete classification of conformal biharmonic maps between any two space forms by  proving the following

\begin{theorem}\label{MT}
A  conformal map $M^m(c_1)\supseteq U \to M^m(c_2)$ between two space forms is proper biharmonic  if and only if 
\begin{itemize}
\item[(i)] the dimension is $m=4$, 
\item[(ii)] the domain is flat (i.e., $c_1=0$), and 
\item[(iii)]  it is a restriction of a M\"obius transformation
$\phi: \r^4\to (\r^4,h)$, $\phi(x)= b+\frac{kA(x-a)}{|x-a|^{\epsilon}},\; a, b\in \r^4,\; A\in O(4)$, 
 into a $4$-dimensional space form, where $\epsilon =2$ for $c_2=0$ and $\epsilon\in \{0, 2\}$ for $c_2\ne 0$.
 \end{itemize}
\end{theorem}
 We also show that proper $k$-polyharmonic conformal maps between Euclidean spaces exist if and only if the dimension is $2k$ and they are precisely the restrictions of  M\"obius transformations. This gives infinitely many simple examples of proper $k$-polyharmonic maps with very nice geometric structure.

\section{Proof of the main theorem}
First of all, thanks to the unique continuation theorem  for biharmonic maps (see \cite{BOn}) which states that two biharmonic maps are the same if they agree on an open subset of the domain, it is enough to prove the classification theorem in an open subset of the domain.

The proof of the main theorem rests on the following two lemmas. The first one  gives the local expressions of conformal maps between space forms which allow us to compute the conformal factors that satisfy the biharmonic PDEs.  

Recall that a conformally flat space is a Riemannian manifold $(M, g)$ which is locally Euclidean, i.e., every point of $M$ is contained in an open neighborhood $U$ such that $g_U=\lambda^2g_E$, where $g_E$ is the Euclidean metric.

\begin{lemma}\label{Le1}
For $m\ge 3$, any conformal map $\phi: (M^m,g)\to (N^m,h)$ between conformally flat spaces can locally be described as  $\phi: M^m\supseteq U\to \phi(U)\subseteq\r^m$ ($m\ge 3$) with $ \phi(x)= b+\frac{kA(x-a)}{|x-a|^{\epsilon}},\; a, b\in \r^m,\; A\in O(m)$
\end{lemma}
\begin{proof}
Let $\phi: (M^m,g)\to (N^m,h)$ be a conformal map between two conformally flat spaces with conformal factor $\lambda$, then for any point $p\in M$ there exist neighborhood $U$ of $p$ and $V$ of $\phi(p)$ such that its local expression $\tilde{\phi}: (\r^m, \sigma^2\delta_{ij}) \supset U\to (\r^m, \rho^2 \delta_{\alpha, \beta})$ is a conformal map with the conformal factor $\lambda$. It is easily checked that the map $\tilde{\phi}: (\r^m, \delta_{ij}) \supset U\to (\r^m, \delta_{\alpha,\beta})$ between Euclidean spaces is a conformal map with the conformal factor $\sigma \lambda\, \rho^{-1} (\tilde{\phi})$. By the well-known Liouville theorem for conformal map (see e.g., \cite{BW}, Proposition 2.3.14), we conclude that the map has to be in the form $ \tilde{\phi}(x)= b+\frac{kA(x-a)}{|x-a|^{\epsilon}},\; a, b\in \r^m,\; A\in O(m)$.
\end{proof}

The following lemma gives some necessary conditions for a proper biharmonic conformal map between space forms, which are crucial in the proofs of the classification theorems.

\begin{lemma}\label{Le2}
 If a conformal map  $\phi: (M^m(c_1), g)\to (M^m(c_2), h)$  between two space forms with $m\ge 3$  with $\phi^{*}h=\lambda^2 g$ is biharmonic, then the  conformal factor $\lambda$ solves the PDEs
\begin{eqnarray}\label{ND}
2\nabla\,(\lambda\,\Delta \lambda) -4(\Delta \lambda)\nabla \lambda  +[2mc_2\lambda^2+(m-2) c_1]\,\lambda\nabla\lambda=0, \;{\rm and}\\\label{ND2}
 (m-4)\nabla |\nabla \lambda|^2  +[4\Delta \lambda+(2-3m)c_1\lambda+2mc_2\lambda^3]\,\nabla\lambda=0.
\end{eqnarray}
\end{lemma}
\begin{proof}
A biharmonic map equation for conformal maps between manifolds of the same dimension was derived in \cite{BFO}, which can be written in the following form  (see \cite{BO}, also \cite{OC}, Corollary 11.11)
\begin{eqnarray}\label{SDL}
\lambda\,\nabla\,\Delta \lambda -3(\Delta \lambda)\nabla \lambda  -\frac{m-4}{2}\nabla |\nabla \lambda|^2+2\lambda {\rm Ric}^M(\nabla \lambda)=0,
\end{eqnarray}
where $\Delta$ and  $\nabla $ denote the Laplacian  and the gradient of the domain manifold $(M^m, g)$.

On the other hand, it is known (see e.g., \cite{BW}, Proposition 11.4.2, also \cite{OC}, Lemma 11.1) that the conformal factor of a conformal map $\phi: (M^m, g)\to (N^m, h)$ with $m\ge 2$  with $\phi^{*}h=\lambda^2 g$ satisfies the equation
 \begin{eqnarray}\label{CL}
\Delta \lambda=\frac{1}{2(m-1)}(\lambda\, {\rm Scal}^M-\lambda^3\, {\rm Scal}^N)-\frac{m-4}{2\lambda}|\nabla \lambda|^2,
\end{eqnarray}
where, ${\rm Scal}^M$ and ${\rm Scal}^N$ denote the scalar curvatures of $(M^m, g)$ and $(N^m,h)$ respectively. 

In particular, for a  conformal map $\phi :(M^m(c_1), g) \to (N^m(c_2), h)$ between space forms with $\phi^{*}h=\lambda^2g$ and $m\ge 3$, Equations (\ref{SDL}) and (\ref{CL}) reduce to 
\begin{eqnarray}\label{SnD}
& \nabla\,(\lambda\,\Delta \lambda) -4(\Delta \lambda)\nabla \lambda  -\frac{m-4}{2}\nabla |\nabla \lambda|^2+2(m-1)c_1\lambda \nabla \lambda=0,\; {\rm and}\\\label{conf}
& \lambda\Delta \lambda-\frac{m}{2}(c_1\lambda^2-c_2\lambda^4)+\frac{m-4}{2}|\nabla \lambda|^2=0
\end{eqnarray}
respectively.

Applying $\nabla$ to both sides of (\ref{conf}) and adding the results to (\ref{SnD})  gives the following necessary condition for a conformal biharmonic maps between two space forms.
\begin{eqnarray}\notag
2\nabla\,(\lambda\,\Delta \lambda) -4(\Delta \lambda)\nabla \lambda  +[2mc_2\lambda^2+(m-2) c_1]\,\lambda\nabla\lambda=0.
\end{eqnarray}
Applying $-\nabla$ to both sides of (\ref{conf}) and adding the results to (\ref{SnD})  gives another necessary condition for a conformal biharmonic maps between two space forms.
\begin{eqnarray}\notag
 (m-4)\nabla |\nabla \lambda|^2  +[4\Delta \lambda+(2-3m)c_1\lambda+2mc_2\lambda^3]\,\nabla\lambda=0.
\end{eqnarray}
Thus, we obtain the lemma.
\end{proof}

\begin{proposition}\label{RR}
A conformal map $\phi: \r^m\supseteq U\to \r^m$  is proper biharmonic if and only if $m=4$ and it is a restriction of  the M\"obius transformation $\phi: \r^4\to \r^4 $, $\phi(x)= b+\frac{kA(x-a)}{|x-a|^2},\; a, b\in \r^4,\; A\in O(4).
$\end{proposition}
\begin{proof}
First, we note that for $m=2$ any conformal map is harmonic so there is no proper biharmonic conformal map between $2$-dimensional space forms. For $m\ge 3$, it follows from the  well-known Liouville theorem that any conformal map $\phi: \r^m\supseteq U\to \r^m$  is a composition of similarities and the inversions in spheres. It follows that we can express the map as 
\begin{align}\label{homoth}
\phi(x)= b+kA(x-a),\; a, b\in \r^m,\; k\in \r\setminus\{0\}, \; A\in O(m),\; {\rm or}\\\label{Inverse}
\phi(x)= b+\frac{kA(x-a)}{|x-a|^2},\; a, b\in \r^m,\; A\in O(m).
\end{align}
It is easily seen that the map given in (\ref{homoth}) is harmonic. So we only need to check the biharmonicity of the maps given in (\ref{Inverse}). A straightforward computation gives the conformal factor $\lambda$  of $\phi$ as
\begin{equation}\label{LRR}
\lambda=\frac{k}{|x-a|^2}.
\end{equation}
A further computation yields
\begin{align}\notag
\Delta\lambda=&\frac{-2k(m-4)}{|x-a|^4}=-\frac{2}{k}(m-4)\lambda^2,\\\label{GD10}
\nabla(\lambda\Delta\lambda)=&-\frac{6}{k}(m-4)\lambda^2\nabla \lambda,\\\label{GD11}
-2(\Delta\lambda)\nabla \lambda=&\frac{4}{k}(m-4)\lambda^2\nabla \lambda.
\end{align}
Note that for conformal maps between Euclidean domains, Equation (\ref{ND}) reads 
\begin{align}\label{GD12}
\nabla(\lambda\Delta\lambda)-2(\Delta\lambda)\nabla \lambda=0.
\end{align}
Substituting (\ref{GD10}) and (\ref{GD11}) into (\ref{GD12}) we have
\begin{align}
-\frac{2}{k}(m-4)\lambda^2\nabla \lambda=0.
\end{align}
This equation has solution  only if $m=4$ since $\nabla \lambda\ne 0$. 

For $m=4$, it was proved in \cite{BO} that any M\"obius transformation $\r^4 \supseteq U\to \r^4$ defined by (\ref{homoth}) and (\ref{Inverse}) are biharmonic, and only the latter cases are proper biharmonic. Thus, we obtain the proposition.
\end{proof}

\begin{proposition}\label{RS}
A conformal map $\phi: \r^m\supseteq U\to S^m$  is proper biharmonic if and only if $m=4$ and  it is a restriction of  one of the M\"obius transformation  $\phi: \r^4 \to (\r^4, \frac{4\delta_{\alpha\beta}}{(1+|y|^2)^2})$ with $ \phi(x)= b+\frac{kA(x-a)}{|x-a|^{\epsilon}},\; a, b\in \r^4,\; A\in O(4), \epsilon\in \{0, 2\}.$
\end{proposition}
\begin{proof}
Identifying  $S^m\setminus\{N\}$ with $(\r^m, \rho^2\delta_{ij})$, where $\rho^2=\frac{4}{(1+|y|^2)^2}$, we use Lemma \ref{Le1} to conclude that in an open set $U$
the conformal map $\phi: \r^m\supseteq U\to S^m$ takes the forms given by (\ref{homoth}) or(\ref{Inverse}).

A straightforward computation gives the conformal factor $\lambda$  of $\phi$ as
\begin{align}\notag
\lambda=&\frac{2c}{c^2+|x-d|^2},\; {\rm where},\; c=\frac{k}{1+|b|^2}\ne 0,\; d=a-\frac{kA^tb}{1+|b|^2},\\\label{NL}
\nabla \lambda=&\frac{-4c(x-d)}{(c^2+|x-d|^2)^2},\\\label{DL}
\Delta\lambda=&
\frac{-4c[mc^2+(m-4)|x-d|^2]}{(c^2+|x-d|^2)^3},\\\notag
\lambda\Delta\lambda=&
-\frac{m}{2}\lambda^4-\frac{8c^2(m-4)|x-d|^2}{(c^2+|x-d|^2)^4}\\\label{NDL}
\nabla(\lambda\Delta\lambda)=&-2m\lambda^3\nabla \lambda-\frac{8c^2(m-4)(2c^2-6|x-d|^2)(x-d)}{(c^2+|x-d|^2)^5}.
\end{align}

When $ c_1=0, c_2=1$, the necessary condition  (\ref{ND})  for a conformal biharmonic map reads
\begin{eqnarray}\label{01}
\nabla\,(\lambda\,\Delta \lambda) -2(\Delta \lambda)\nabla \lambda  +m \lambda^3\nabla \lambda=0.
\end{eqnarray}
Substituting (\ref{NL}), (\ref{DL}) and (\ref{NDL}) into  (\ref{01}) we have
\begin{eqnarray}\notag
-m\lambda^3\nabla \lambda-\frac{8c^2(m-4)(2c^2-6|x-d|^2)(x-d)}{(c^2+|x-d|^2)^5}\\\notag-2\frac{-4c[mc^2+(m-4)|x-d|^2]}{(c^2+|x-d|^2)^3}\frac{-4c(x-d)}{(c^2+|x-d|^2)^2}=0
\end{eqnarray}
This is equivalent to
\begin{align}\notag
&32mc^4-8c^2(m-4)(2c^2-6|x-d|^2)-32c^2[mc^2+(m-4)|x-d|^2]=0,\; {\rm or}\\\notag
&-16c^2(m-4)(c^2-|x-d|^2)=0.
\end{align}
This equation has solution  only if $m=4$. 
For $m=4$, it was proved in \cite{BO} that any M\"obius transformation $\r^4 \supseteq U\to S^4$ defined by (\ref{homoth}) and (\ref{Inverse}) are proper biharmonic.  Thus, we obtain the proposition.
\end{proof}

\begin{proposition}\label{RH}
A conformal map $\phi: \r^m\supseteq U\to H^m$ ($m\ge 3$) is proper biharmonic if and only if $m=4$ and  it is a restriction of  a M\"obius transformation  $\phi: \r^4 \to (B^4, \frac{4\delta_{\alpha\beta}}{(1-|y|^2)^2})$ with $ \phi(x)= b+\frac{kA(x-a)}{|x-a|^{\epsilon}},\; a, b\in \r^4,\; A\in O(4), \epsilon\in \{0, 2\}.$
\end{proposition}
\begin{proof}
We omit the proof of the fact that a proper biharmonic conformal map $\phi: \r^m\supseteq U\to H^m$ ($m\ge 3$) exists only in the case of dimension $m=4$ which  is similar to  that of Proposition \ref{RS}. For the second statement, we check that the conformal factor of the conformal map $\phi: \r^4\supseteq U\to H^4\equiv (B^4, \frac{4\delta_{\alpha\beta}}{1-|y|^2)^2})$ with $\phi(x)= b+\frac{kA(x-a)}{|x-a|^{\epsilon}},\; a, b\in \r^4,\; A\in O(4), \epsilon\in \{0, 2\}$,  is given by
\begin{align}\notag
\lambda=&\frac{2c}{-c^2+|x-d|^2},\; {\rm where},\; c=\frac{k}{1-|b|^2}\ne 0,\; d=a+\frac{kA^tb}{1-|b|^2}.
\end{align}
A straightforward computation verifies that the conformal factor  satisfies
\begin{align}\notag
\Delta\lambda=2\lambda^3,
\end{align}
which is exactly the biharmonic equation (\cite{BO}) for a conformal map $\r^4\to H^4\equiv (B^4, \frac{4\delta_{\alpha\beta}}{1-|y|^2)^2})$.
\end{proof}

\begin{proposition}\label{HR}
(i) There exists no proper biharmonic  conformal map $\phi: H^m \supseteq U\to \r^m$;
(ii) There exists no proper biharmonic  conformal map $\phi: S^m \supseteq U\to \r^m$.
\end{proposition}
\begin{proof}
Use the model  $(B^m, \sigma^2\delta_{ij})$, $\sigma^2=\frac{4}{(1-|x|^2)^2}$ for the domain space $H^m$. By the unique continuation theorem for biharmonic maps, it is enough to prove the proposition in an open subset. By Lemma \ref{Le1},  $\phi: H^m\supseteq U\to \r^m$ ($m\ge 3$)  can be expressed as 
\begin{align}\label{h}
\phi(x)= b+\frac{kA(x-a)}{|x-a|^{\epsilon}},\; a, b\in \r^m,\; A\in O(m), \epsilon\in \{0, 2\}.
\end{align}
Let $\bar g=\sigma^2 g$ with $\sigma= \frac{2}{1-|x|^2}$ and $g$ denoting the standard Euclidean metric.  Let $\bar \Delta$ and $\bar \nabla$ denote the Laplacian and the gradient operators taken with respect to $\bar g$ respectively, and  $ \Delta$ and $ \nabla$ be the corresponding operators with respect to the Euclidean metric. Then, one can check that  the conformal factor of the map $\phi: H^m\supseteq U\to \r^m$ is given by
\begin{align}\notag
\lambda=\frac{k(1-|x|^2)}{2|x-a|^2}=\frac{k}{2}(1-|x|^2)f^{-1},\;  
\end{align}
 where,  $f=|x-a|^2$ and $k={\rm constant}\ne 0$.\\
 
A further computation yields
\begin{align}\notag
&\bar \nabla \lambda=-k\sigma^{-2}f^{-2} [ (f+(1-|x|^2))x-(1-|x|^2)\,a\,],\\\notag
&|\bar \nabla \lambda|^2=k^2\sigma^{-2}f^{-4}\{(1+f-|x|^2)^2 |x|^2+(1-|x|^2)^2|a|^2\\\notag
& -2 \langle x, a\rangle (1-|x|^2)(1+f-|x|^2)\},\\\label{HRL}
&\bar \Delta \lambda=-k\sigma^{-1}f^{-3}\{\sigma^{-1}[ mf^2-[( 4+m)|x|^2-4\langle x, a\rangle-m]f\\\notag 
&-4(1-|x|^2)|x-a|^2]\\\notag
&+(m-2)f\,[(1+f-|x|^2)|x|^2-(1-|x|^2)\langle x, a\rangle]\},\\\label{HRN}
&\nabla |\bar \nabla \lambda|^2=k^2\sigma^{-1}f^{-5}\,x\;\{(-2 f -8\sigma^{-1})[1+f-|x|^2)^2 |x|^2+(1-|x|^2)^2|a|^2\\\notag
&-2 \langle x, a\rangle (1-|x|^2)(1+f-|x|^2)]+\sigma^{-1}f\,[ 2(1+f-|x|^2)^2\\\notag
&-4|a|^2(1-|x|^2)+4(1+f-|x|^2)\langle x, a\rangle ]\,\}\\\notag
&+k^2\sigma^{-1}f^{-5}\,a\, \{8\sigma^{-1} [1+f-|x|^2)^2 |x|^2+(1-|x|^2)^2|a|^2\\\notag
&-2 \langle x, a\rangle (1-|x|^2)(1+f-|x|^2)]+\sigma^{-1}f\,[-4(1+f-|x|^2)|x|^2 \\\notag
&-2(1-|x|^2)(1+f-|x|^2) +4(1-|x|^2)\langle x, a\rangle]\,\}.
\end{align}

Note that for $c_1= -1, c_2=0$, Equation (\ref{ND2}) can be written as
\begin{eqnarray}\label{HR2}
 (m-4)\nabla |\bar\nabla \lambda|^2  +[4\bar \Delta \lambda-(2-3m)\lambda]\,\nabla\lambda=0.
\end{eqnarray}

Substituting (\ref{HRL}) and (\ref{HRN}) into (\ref{HR2}) we have
\begin{align}\label{GD28}
P_1(x)x+P_2(x)a=0
\end{align}
for any $x$, where $P_1(x)$ and $P_2(x)$ are two polynomials in $x$.

Case 1: $a=0$. In this case,  Equation (\ref{GD28}) implies that  $P_1(x)=0$ which leads to a polynomial equation
\begin{align}
&(m-4)\{(-2 |x|^2 -8\sigma^{-1}) |x|^2+2\sigma^{-1}|x|^2\, \,\}\\\notag
+& 4\sigma^{-1}[ m|x|^4-(( 4+m)|x|^2-m)|x|^2-4(1-|x|^2)|x|^2]\\\notag
&+4(m-2)\,|x|^4+(2-3m)|x|^4 =0.
\end{align}
A further calculation yields
\begin{align}
&-2 |x|^4 -(m-4)|x|^2 =0,\;\forall\; x,
\end{align}
which has no solution.

Case 2: $a\ne 0$. In this case, we choose any $x\in \r^m$ with $\langle x, a\rangle=0$, then (\ref{GD28}) implies that $P_2(x)=0, \;\forall\; x\,\bot\, a$.

Using this, together with (\ref{HRL}), (\ref{HRN}), (\ref{HR2}), and a straightforward computation, we obtain 

\begin{align}
&(m-4)\, \{4 [A^2 |x|^2+(1-|x|^2)^2|a|^2]+(|x|^2+|a|^2)\,(-A-A|x|^2) ]\,\}\\\notag
&- 2(1-|x|^2)[ m(|x|^2+|a|^2)^2-(( 4+m)|x|^2-m)(|x|^2+|a|^2)\\\notag
&-4(1-|x|^2)(|x|^2+|a|^2)]\\\notag
&-4(m-2)A(|x|^2+|a|^2)\,|x|^2-(2-3m)(|x|^2+|a|^2)^{2}=0
\end{align}
where $A=1+|a|^2$.

By computing the coefficients of the constant term and the $|x|^2$ term, we have
\begin{align}\label{CC}
&(m-4) +2|a|^2=0,\\\label{2C}
&(m-4)+(m-4)|a|^4+4|a|^2=0.
\end{align}
The only possible solution of (\ref{CC}) and (\ref{2C}) is $m=2$. However, it is well known that in the case of $m=2$, any conformal map between $2$-dimensional manifolds is harmonic. Thus, we obtain the first statement of the proposition.\\

Using a similar way working with conformal factor $\lambda=\frac{k(1+|x|^2)}{2|x-a|^2}$ we obtain the second statement.
\end{proof}

\begin{proposition}\label{SS}
(i) There exists no proper biharmonic  conformal map $\phi: S^m \supseteq U\to S^m$; (ii) There exists no proper biharmonic  conformal map $\phi: H^m \supseteq U\to H^m$
\end{proposition}
\begin{proof}
 Identifying  the domain sphere $S^m\setminus\{N\}$ with $(\r^m, \sigma^2\delta_{ij})$, $\sigma^2=\frac{4}{(1+|x|^2)^2}$, and the target sphere with $(\r^m, \frac{4\delta_{\alpha\beta}}{(1+|y|^2)^2})$. By the unique continuation theorem for biharmonic maps, it is enough to prove the proposition for biharmonic conformal maps $\phi:  (\r^m, \frac{4\delta_{ij}}{(1+|x|^2)^2})\supset U\to(\r^m, \frac{4\delta_{\alpha\beta}}{(1+|y|^2)^2})$. By Lemma \ref{Le1},  $\phi: \r^m\supseteq U\to \r^m$ ($m\ge 3$)  can be expressed as 
\begin{align}\label{h}
\phi(x)= b+\frac{kA(x-a)}{|x-a|^{\epsilon}},\; a, b\in \r^m,\; A\in O(m), \epsilon\in \{0, 2\}.
\end{align}
Let $\bar g=\sigma^2 g$ with $\sigma= \frac{2}{1+|x|^2}$ and $g$ denoting the standard Euclidean metric.  Let $\bar \Delta$ and $\bar \nabla$ denote the Laplacian and the gradient operators taken with respect to $\bar g$ respectively, and  $ \Delta$ and $ \nabla$ be the corresponding operators with respect to the Euclidean metric. Then, one can check that  the conformal factor of the map between spheres is given by
\begin{align}\notag
\lambda=\frac{c(1+|x|^2)}{c^2+|x-d|^2},\; {\rm where},\; c=\frac{k}{1+|b|^2}\ne 0, \;d=a-\frac{kA^tb}{1+|b|^2}. 
\end{align}

Denoting $F=c^2+|x-d|^2$, a further computation yields
\begin{align}\notag
\bar \nabla \lambda=&2c\sigma^{-2}F^{-2} [ (F-(1+|x|^2))x+(1+|x|^2)\,d\,],\\\notag
|\bar \nabla \lambda|^2=&4c^2\sigma^{-2}\,[F^{-2} |x|^2+(1+|x|^2)^2F^{-4}(|x|^2+|d|^2\\\notag
&-2 \langle x, d\rangle)-2(1+|x|^2)F^{-3}(|x|^2-\langle x, d\rangle)],\\\notag
\bar \Delta \lambda=&\sigma^{-2}\Delta \lambda+(m-2)\sigma^{-3}\langle \nabla \sigma, \nabla \lambda\rangle\\\notag
=&2c\sigma^{-2}F^{-3}\{ mF^2-[( 4+m)|x|^2-4\langle x, d\rangle+m]F+4(1+|x|^2)|x-d|^2\}\\\notag
&-2c(m-2)\sigma^{-1}F^{-3}\{ |x|^2F^2-(1+|x|^2)(|x|^2-\langle x, d\rangle) F\}.
\\\label{GD30}
&(m-4)\nabla |\bar \nabla \lambda|^2=4(m-4)c^2F^{-5}\,x\\\notag
&\{2\sigma^{-1}\,[ |x|^2F^{3}+(1+|x|^2)^2|x-d|^2F-2(1+|x|^2)(|x|^2-\langle x, d\rangle)F^2]\\\notag
&+\sigma^{-2}[2F^3-4|x|^2F^2+4(1+|x|^2)|x-d|^2F-8(1+|x|^2)^2|x-d|^2\\\notag &+2(1+|x|^2)^2F
-4F^2(|x|^2-\langle x, d\rangle) -4(1+|x|^2)F^2\\\notag &+12(1+|x|^2)F(|x|^2-\langle x, d\rangle)]\}\\\notag
&+4(m-4)c^2\sigma^{-2}F^{-5}d\;[4|x|^2F^2
+8(1+|x|^2)^2|x-d|^2-2(1+|x|^2)^2F\\\notag
&-12(1+|x|^2)F(|x|^2-\langle x, d\rangle)+2(1+|x|^2)F^2].
\\\label{GD31}
 &[4\bar\Delta \lambda+(2-3m)\lambda+2m\lambda^3]  \nabla\lambda=2c^2F^{-5} [ c^2+|d|^2-1-2\langle x,d\rangle\,]x\\\notag
&\{8\sigma^{-2}[ mF^2-[( 4+m)|x|^2-4\langle x, d\rangle+m]F+4(1+|x|^2)|x-d|^2]\\\notag
&-8(m-2)\sigma^{-1}[ |x|^2F^2-(1+|x|^2)(|x|^2-\langle x, d\rangle) F]\\\notag
&+[(2-3m)(1+|x|^2)F^2+2mc^2 (1+|x|^2)^3]\}\;+2c^2F^{-5}  (1+|x|^2)\,d\\\notag
&\{8\sigma^{-2}[ mF^2-(( 4+m)|x|^2-4\langle x, d\rangle+m)F+4(1+|x|^2)|x-d|^2]\\\notag
&-8(m-2)\sigma^{-1}[ |x|^2F^2-(1+|x|^2)(|x|^2-\langle x, d\rangle) F]\\\notag
&+(2-3m)(1+|x|^2)F^2+2mc^2 (1+|x|^2)^3\}.
\end{align}

Substituting (\ref{GD30}) and (\ref{GD31}) into (\ref{ND2}) we have
\begin{align}\label{GD32}
P_1(x)x+P_2(x)d=0
\end{align}
for any $x$, where $P_1(x)$ and $P_2(x)$ denote two polynomials in $x$.

Case 1: $d=0$. In this case,  Equation (\ref{GD32}) implies that  $P_1(x)=0$, which leads to a polynomial identity
\begin{align}\notag
&(m-4)\{ 2 |x|^2F^{3}+2(1+|x|^2)^2|x|^2F-4(1+|x|^2)|x|^2F^2\\\notag
&+ (1+|x|^2)[F^3-2|x|^2F^2+2(1+|x|^2)|x|^2F\\\notag
&-4(1+|x|^2)^2|x|^2+(1+|x|^2)^2F-2F^2|x|^2+6(1+|x|^2)F|x|^2-2(1+|x|^2)F^2]\}\\\notag
&+ ( c^2-1)\{2(1+|x|^2)[ mF^2-(( 4+m)|x|^2+m)F+4(1+|x|^2)|x|^2]\\\notag
&-4(m-2) [ |x|^2F^2-(1+|x|^2)|x|^2 F]+(2-3m)F^2+2mc^2 (1+|x|^2)^2\}=0.
\end{align}
By computing the constant term and the $|x|^2$ term we obtain, respectively,
\begin{align}\label{CT}
&c^2(c^2-1)[-2c^2-(m-4)]=0\\\label{2T}
&(c^2-1)[(m-4)c^4-4(m-3)c^2+m-4]=0.
\end{align}
Noting that $c\ne 0$ and $c^2-1\ne 0$ since $\nabla \lambda= (c^2-1)x\ne 0$, Equations (\ref{CT}) and (\ref{2T}) can be simplified as
\begin{align}\label{S1}
&-2c^2-(m-4)=0,\\\label{S2}
&(m-4)c^4-4(m-3)c^2+m-4=0.
\end{align}
Substituting (\ref{S1}) into (\ref{S2}) we have
\begin{align}
(m-4)[c^4+2(m-3)+1]=0,
\end{align}
which implies that for $m\ge 3$, the only possible solution is $m=4$. However, Equation (\ref{S1}) implies that  $m=4$ leads to a contradiction as $c\ne 0$.
Therefore, the proposition is proved in the case of $d=0$.

Case 2:  $d\ne 0$. In this case, we choose any $x\in \r^m$ with $\langle x, d\rangle=0$, then (\ref{GD32}) implies that $P_2(x)=0, \;\forall\; x\,\bot\, d$.

Using this, together with (\ref{GD30}), (\ref{GD31}), (\ref{ND2}), and a straightforward computation we obtain 
\begin{align}\notag
&(m-4)\;[2|x|^2(A+|x|^2)^2+4(1+|x|^2)^2 (|x|^2+|d|^2)-(1+|x|^2)^2(A+|x|^2)\\\notag
&-6(1+|x|^2)(A+|x|^2)|x|^2+(1+|x|^2)(A+|x|^2)^2]\\\notag
&+2(1+|x|^2)[ m(A+|x|^2)^2-(( 4+m)|x|^2+m)(A+|x|^2)+4(1+|x|^2) (|x|^2+|d|^2)]\\\notag
&-4(m-2)[ |x|^2(A+|x|^2)^2-(1+|x|^2)|x|^2 (A+|x|^2)]\\\notag
&+(2-3m)(A+|x|^2)^2+2mc^2 (1+|x|^2)^2=0,\; \forall\; x\,\bot\,d.
\end{align}
By computing the coefficients of the constant, the $|x|^2$, and the $|x|^4$ terms of this equation we obtain, respectively
\begin{align}\label{D9}
&-2A^2-(m-4)(c^2-|d|^2) =0\\\label{D10}
&\;(m-4)A^2+(m-4)-4|d|^2-4(m-3)c^2=0\\\label{D11}
&-2-(m-4)(c^2-|d|^2)=0.
\end{align}
Equations (\ref{D9}) and (\ref{D11}) implies that $A^2=1$ or $A=1$ since $A=c^2+|d|^2>0$.
Note that if $A=1$, then a straightforward computation shows that at the point where $\langle x, d\rangle=0$,
\begin{align}
\nabla |\bar \nabla \lambda|^2=8c^2|d|^2(A-1)\frac{(1+|x|^2)^3}{(A+|x|^2)^5}\,x=0.
\end{align}
From this and (\ref{GD30}) we have
\begin{align}\label{GD33}
&4|x|^2F^2
+8(1+|x|^2)^2|x-d|^2-2(1+|x|^2)^2F\\\notag
&-12(1+|x|^2)F(|x|^2-\langle x, d\rangle)+2(1+|x|^2)F^2=0.
\end{align}
Computing the constant term of this polynomial equation yields
\begin{align}
8|d|^2-2A+2A^2=0,
\end{align}
which, together with $A=1$, implies that $8|d|^2=0$, and hence $d=0$,  a contradiction.
Combining the results in Case 1 and Case 2 we obtain Statement (i).

By a similar method working with the conformal factor $\lambda=\frac{c(1-|x|^2)}{-c^2+|x-d|^2}$ gives Statement (ii).
\end{proof}
\begin{proposition}\label{SH}
(i) There exists no proper biharmonic  conformal map $\phi: S^m \supseteq U\to H^m$; (ii) There exists no proper biharmonic  conformal map $\phi: H^m \supseteq U\to S^m$
\end{proposition}
\begin{proof}
The proofs are similar to that of Proposition \ref{SS}. Here is a sketch for Statement (i).
 Identifying  the domain sphere $S^m\setminus\{N\}$ with $(\r^m, \sigma^2\delta_{ij})$, $\sigma^2=\frac{4}{(1+|x|^2)^2}$, and the target hyperbolic space with $(\r^m, \frac{4\delta_{\alpha\beta}}{(1-|y|^2)^2})$. Use the local expression of  a conformal maps $\phi:  (\r^m, \frac{4\delta_{ij}}{(1+|x|^2)^2})\supset U\to(\r^m, \frac{4\delta_{\alpha\beta}}{(1-|y|^2)^2})$
 \begin{align}\label{h}
\phi(x)= b+\frac{kA(x-a)}{|x-a|^{\epsilon}},\; a, b\in \r^m,\; A\in O(m), \epsilon\in \{0, 2\}
\end{align}
to compute the conformal factor as
\begin{align}\notag
\lambda=\frac{c(1+|x|^2)}{-c^2+|x-d|^2},\; {\rm where},\; c=\frac{k}{1-|b|^2}\ne 0, \;d=a+\frac{kA^tb}{1-|b|^2}. 
\end{align}
Using (\ref{ND2}) with $c_1=1, c_2=-1$ and straightforward computation to have
\begin{align}\label{SH32}
P_1(x)x+P_2(x)d=0
\end{align}
for any $x$, where $P_1(x)$ and $P_2(x)$ denote two polynomials in $x$.

For $d=0$.  Equation (\ref{SH32}) implies that  $P_1(x)=0$, whose constant term and the $|x|^2$ term lead to
\begin{align}\label{SHCT}
&c^2(c^2+1)[2c^2-(m-4)]=0\\\label{SHC}
&-(c^2+1)\{(m-4)c^4+4(m-3)c^2+m-4\}=0,
\end{align}
which have no solution.

For $d\ne 0$, choose any $x\in \r^m$ with $\langle x, d\rangle=0$, then (\ref{SH32}) implies that $P_2(x)=0, \;\forall\; x\,\bot\, d$.

By computing the coefficients of the constant, the $|x|^2$, and the $|x|^4$ terms of this equation we obtain, respectively
\begin{align}\label{SH4}
&-2A^2+(m-4)(c^2+|d|^2) =0\\\label{SH5}
&(m-4)A^2+m-4-4|d|^2+4(m-3)c^2=0,\\\label{SH6}
&(m-4)(c^2+|d|^2)-2 =0
\end{align}
Equations (\ref{SH4}) and (\ref{SH6}) implies that $A^2=1$ or $A=\pm 1$. A straightforward checking shows that $A=-1$ leads to $m=2$, in which case, we know that there is no proper conformal biharmonic map.

For $A=1$, the argument used in the proof of Proposition \ref{SS} can be carried over verbatim. 
The proof of Statement (ii) is similar and is omitted.
\end{proof}

Summarizing  the classifications in Propositions \ref{RR}, \ref{RS}, \ref{RH}, \ref{HR}, \ref{SS}, and \ref{SH} we obtain Theorem \ref{MT}.

\begin{remark}
It is well known that conformally flat spaces are a generalization of space forms. We would like to point out that our classification theorem for proper biharmonic conformal maps between space forms are not generalized to conformal maps between conformally flat spaces. For example, it was proved in \cite{BO} that  there is an infinite family of conformally flat metrics on $S^4$ such that the conformal map given by the identity map $I: (S^4, g_{can})\to (S^4, \lambda^2 g_{can})$ is a proper biharmonic map from the standard sphere (a non-flat conformally flat space) to a conformally flat sphere.
\end{remark}

\section{Classification of $k$-polyharmonic conformal maps between Euclidean spaces}
 $k$-Polyharmonic maps ( or polyharmonic maps of order $k$) are generalization of both harmonic maps and biharmonic maps. They can be characterized as maps between Riemannian manifolds with vanishing $k$-tension field which, depending on $k=2s$ or $k=2s+1$, are given by
\begin{align}
\tau_{2s}(\phi)=&\bar{\Delta}^{2s-1}\tau(\phi)-R^N(\bar{\Delta}^{2s-2}\tau(\phi), d\phi (e_i))d\phi(e_i)\\\notag
&-\sum_{l=1}^{s-1}\{ R^N(\nabla^{\phi}_{e_i}\bar{\Delta}^{s+l-2}\tau(\phi), \bar{\Delta}^{s-l-1}\tau(\phi))d\phi(e_i)\\\notag
&-R^N(\bar{\Delta}^{s+l-2}\tau(\phi), \nabla^{\phi}_{e_i}\bar{\Delta}^{s-l-1}\tau(\phi))d\phi(e_i)\},\;\;s=1, 2, 3\cdots,\\\notag
\tau_{2s+1}(\phi)=&\bar{\Delta}^{2s}\tau(\phi)-R^N(\bar{\Delta}^{2s-1}\tau(\phi), d\phi (e_i))d\phi(e_i)\\\notag
&-\sum_{l=1}^{s-1}\{ R^N(\nabla^{\phi}_{e_i}\bar{\Delta}^{s+l-1}\tau(\phi), \bar{\Delta}^{s-l-1}\tau(\phi))d\phi(e_i)\\\notag
&-R^N(\bar{\Delta}^{s+l-1}\tau(\phi), \nabla^{\phi}_{e_i}\bar{\Delta}^{s-l-1}\tau(\phi))d\phi(e_i)\},\\\notag
&-R^N(\nabla^{\phi}_{e_i}\bar{\Delta}^{s-1}\tau(\phi), \bar{\Delta}^{s-1}\tau(\phi))d\phi(e_i),\;\;\; s=0, 1, 2, \cdots,
\end{align}
where  $ \bar{\Delta}=-\sum_{i=1}^m(\nabla^{\phi}_{e_i}\nabla^{\phi}_{e_i}-\nabla^{\phi}_{\nabla^M_{e_i}e_i})$ with $ \bar{\Delta}^{-1}=0$, and $\{e_i\}$ is an orthonormal from on $M$. The curvature convention is that for $R^{N(c)}(X, Y)Z=c(\langle Y, Z\rangle X- \langle X, Z\rangle Y)$.
 
 Clearly, $k$-polyharmonic map equations are much more complicated than that of biharmonic maps, so few examples of $k$-polyharmonic maps have been found. For some recent work on $k$-polyharmonic maps see \cite{Wa1, Wa2, Ma1, Ma2, Ma3, Ma4, MR1, MR2, MOR2, NU, Br, Br2, BMOR, BMOR2}.
 
 In this section we give a complete classification of $k$-polyharmonic conformal maps between Euclidean spaces, which provides infinitely many simple examples of proper $k$-polyharmonic maps with nice geometric structure.
 
It follows from $k$-tension field that any harmonic map is a $k$-polyharmonic map for $k\ge 2$ so we will call those $k$-polyharmonic maps which are not harmonic {\em proper $k$-polyharmonic maps}. Note also that in general, a $k$-polyharmonic map for $k\ge 2$ need not be a $(k+1)$-polyharmonic map.

\begin{corollary}
$\phi: (M^m,g)\to \r^n$ is a k-polyharmonic map into Euclidean space if and only if $\Delta^k\phi=0$, where $\Delta$ is the Laplacian on $(M, g)$. In particular, a $k$-polyharmonic map into a Euclidean space is alway a $(k+1)$-polyharmonic map for any $k=1, 2,\cdots$.
\end{corollary}

\begin{theorem}\label{MT2}
A conformal map   $\phi: \r^m\supseteq U\to \r^m$  is a proper $k$-polyharmonic map if and only if $m=2k$ for $k\ge 2$ and it is a restriction of the M\"obius transformation $\phi: \r^m\setminus\{0\} \to \r^m,\;\phi(x)=b+\frac{cA(x-a)}{|x-a|^2}$. In particular,  there is no  proper $k$-polyharmonic conformal maps (for any $k\ge 2$) between domains of an odd dimensional Euclidean space.
\end{theorem}
\begin{proof}
For $m=2$, it is well known that any conformal map $\phi: \r^2\supseteq U\to \r^2$ is harmonic. For $m\ge 3$, the Liouville theorem for conformal maps implies that any conformal map $\phi: \r^m\supseteq U\to \r^m$ takes the form $\phi(x)=b+\frac{cA(x-a)}{|x-a|^{\epsilon}}$, where $\epsilon=\{0, 2\}, a, b\in\r^m, A\in O(m)$. For $\epsilon=0$ the conformal map reduces to a homothety which is always a harmonic map. For $\epsilon=2$,  $m\ge 3$, and $i\in \{ 1, 2, \cdots, m\}$ fixed, a straightforward computation yields

\begin{align}\notag
\Delta \left( \frac{x_i-a_i}{|x-a|^2}\right)=&-2(m-2) \frac{x_i-a_i}{|x-a|^4},\\\notag
\Delta^2 \left( \frac{x_i-a_i}{|x-a|^2}\right)=&(-2)(-4)(m-2)(m-4) \frac{x_i-a_i}{|x-a|^6},\\\notag
\cdots,\;\;\cdots,\\\notag
\Delta^k \left( \frac{x_i-a_i}{|x-a|^2}\right)=&(-)^k2\cdot 4\cdots (2k)(m-2)(m-4)\cdots (m-2k) \frac{x_i-a_i}{|x-a|^{2(k+1)}}.
\end{align}
It follows that
\begin{align}\notag
&\Delta^k \phi=\Delta^k \left(b+\frac{cA(x-a)}{|x-a|^2}\right)\\\notag
&=(-1)^k2\cdot 4\cdot 6\cdot \cdots \cdot (2k) (2-m)\cdot (4-m)\cdot (6-m)\cdot \cdots \cdot (2k-m)\frac{cA(x-a)}{|x-a|^{2+2k}},
\end{align}
from which we obtain the theorem.
\end{proof}
As a straightforward consequence of Theorem \ref{MT2}, we have 
\begin{corollary}\label{Co}
For any $k\ge 2$, the inversion in the sphere  $S^{2k-1}$, $\phi:\r^{2k}\setminus\{0\}\to \r^{2k}$, $ \phi(x)=\frac{x}{|x|^2}$ is a proper $k$-polyharmonic map.
\end{corollary}

\begin{remark}
Note that a $k$-polyharmonic map with $k=2$ is nothing but a biharmonic map, so Corollary \ref{Co} includes the  well-known result of Baird-Kamissoko (the $k=2$ case) as a special case.
\end{remark}

\end{document}